\documentclass{birkmult}
\usepackage{amsmath}
 \newtheorem{thm}{Theorem}[section]
 \newtheorem{cor}[thm]{Corollary}
 \newtheorem{lem}[thm]{Lemma}
 
 \theoremstyle{definition}
 \newtheorem{defn}[thm]{Definition}
 \theoremstyle{remark}
 \newtheorem{rem}[thm]{Remark}
 \newtheorem*{ex}{Example}
 \numberwithin{equation}{section}

\begin{document}
%
%
%
%
%
%
%
%
%
\title[Optimal Decompositions of Translations of $L^{2}$-functions]
 {Optimal Decompositions of Translations of $L^{2}$-functions}
\author[Palle E.T. Jorgensen]{Palle E.T. Jorgensen}

\address{Department of Mathematics\\
The University of Iowa\\
14 MacLean Hall\\
Iowa City, IA 52242}

\email{jorgen@math.uiowa.edu}

\thanks{Work supported in part by the U.S. National Science
Foundation}
\author{Myung-Sin Song}
\address{Department of Mathematics and Statistics\\
Southern Illinois University Edwardsville\\
Campus Box 1653, Science Building\\
Edwardsville, IL 62026}
\email{msong@siue.edu}


\date{Oct 19, 2007}
\dedicatory{To Guido Weiss}

\begin{abstract}
In this paper we offer a computational approach to the spectral function for a finite family of commuting operators, and give applications. 
Motivated by questions in wavelets and in signal processing, we study a 
problem about spectral concentration of integral translations of functions 
in the Hilbert space $L^{2}(\mathbb{R}^{n})$. Our approach applies more 
generally to families of $n$ arbitrary commuting unitary operators in a 
complex Hilbert space $\mathcal{H}$, or equivalent the spectral theory of 
a unitary representation $U$ of the rank-$n$ lattice $\mathbb{Z}^{n}$ in 
$\mathbb{R}^{n}$. Starting with a non-zero vector $\psi \in \mathcal{H}$, 
we look for relations among the vectors in the cyclic subspace in 
$\mathcal{H}$ generated by $\psi$. Since these vectors 
$\{U(k)\psi | k \in \mathbb{Z}^{n}\}$
involve infinite ``linear combinations," the problem arises of giving 
geometric characterizations of these non-trivial linear relations. A 
special case of the problem arose initially in work of Kolmogorov under 
the name $L^{2}$-independence. This refers to \textit{infinite} linear 
combinations of integral translates of a fixed function with 
$l^{2}$-coefficients.  While we were motivated by the study of translation 
operators arising in wavelet and frame theory, we stress that our present 
results are general; our theorems are about spectral densities for general 
unitary operators, and for stochastic integrals.
\end{abstract}

\maketitle
\tableofcontents

\section{Introduction}
\label{sec:1}

The study of frames in Hilbert space serves as an oprator theoretic 
context for a variety of wavelet problems (see e.g., \cite{Dau92}), and 
more generally for the study of generalized bases in harmonic analysis. 
Loosely speaking, a frame is generalized basis system involving the 
kind of ``over completeness." Since many wavelet constructions do not 
yield orthogonality, the occurrence of wavelet frames is common. In this 
paper, we will consider various spectral problems in frame and wavelet 
theory in the context of a single unitary operator $T$ in a fixed Hilbert 
space $\mathcal{H}$, or more generally a finite commuting family of 
unitaty operators. This of course is equivalent with the consideration of 
a unitary representation of the group $\mathbb{Z}$ of all integers, or of 
$\mathbb{Z}^{d}$ in the Hilbert space $\mathcal{H}$. Since every such 
representation is unitarily equivalent to a representation by 
multiplication. We will take advantage of this fact, and account for over 
completeness in terms of spectrum in the sense of spectral multiplicity 
for unitary operators, see e.g., \cite{Hel64, Hel86}. For general facts on 
operators in Hilbert space, the reader may refer to e.g., \cite{Con90}. 

An arbitrary infinite configuration of vectors $(f_{k})_{k \in \mathbb{Z}}$ 
in a Hilbert space $\mathcal{H}$ can be quite complicated, and it will be 
difficult to make sense of finite and infinite linear combinations 
$\sum_{k \in \mathbb{Z}}c_{k}f_{k}$.  However, in applications, a 
particular system of vectors $f_{k}$ may often be analyzed with the use of 
a single unitary operator $U$ in $\mathcal{H}$.  This happens if there is a 
fixed vector $\varphi \in \mathcal{H}$ such that $f_{k} = U^{k}\varphi$ 
for all $k \in \mathbb{Z}$.  When this is possible, the spectral theorem 
will then apply to this unitary operator.  A key idea in our paper is to 
identify a spectral density function computed directly from the pair 
$(\varphi, U)$.  

Hence the study of linear expressions $\sum_{k}c_{k}f_{k}$ may be done 
with the aid of the spectral function for this pair $(\varphi, U)$.  A 
spectral function for a unitary operator $U$ is really a system of 
functions $(p_{\varphi})$, one for each cyclic subspace 
$\mathcal{H}(\varphi)$.  In each cyclic subspace, the function $p_{\varphi}$ 
is a complete unitary invariant for $U$ restricted to 
$\mathcal{H}(\varphi)$: by this we mean that the function $p_{\varphi}$ 
encodes all the spectral data coming from the vectors $f_{k}=U^{k}\varphi$, 
$k \in \mathbb{Z}$. For background literature on the spectral function and 
its applications we refer to \cite{BH05, JP05, LWW04, PSWX04, Sad06, TT07}. 

       In summary, the spectral representation theorem is the assertion 
that commuting unitary operators in Hilbert space may be represented as 
multiplication operators in an $L^{2}$-Hilbert space. The understanding 
is that this representation is defined as a unitary equivalence, and 
that the $L^{2}$-Hilbert space to be used allows arbitrary measures, and 
$L^{2}$ will be a Hilbert space of vector valued functions, 
see e.g., \cite{Hel86}.

      The term ``frame" refers to a generalized basis system involving the 
kind of ``over completeness" that arises in signal processing problems 
with redundancy, see e.g., \cite{BCHL06, CKS06} and \cite{Chr03}. It is our 
aim here to account for this kind of over completeness via representation 
by operators and vectors arising from an appropriately chosen spectral 
representation. 

      While our theorems apply to the general case of the spectral function 
for a finite family of commuting operators in Hilbert space, our motivation 
is from the case of wavelet operators in the Hilbert spaces 
$L^{2}(\mathbb{R})$ and $L^{2}(\mathbb{R}^{d})$. Hence section \ref{sec:2} 
begins with function theory. The general and axiomatic setting is introduced 
in sect 3, which also contains the main technical lemmas. The remaining of 
our paper deals with applications, and they are divided into three parts: 
operator theory, frame theory and wavelet resolution spaces 
(sect \ref{sec:4}), and finally stochastic integration (sect \ref{sec:5}). 
But our recurrent theme is the spectral function.

\subsubsection*{Notation.} We will be using $\mathbb{Z}$ for the group of 
integers, $\mathbb{R}$ for the real numbers, $\mathbb{C}$ for the 
complexes, and $\mathbb{T}$ for the circle, or equivalently the 
one-torus, $\mathbb{T} = \{z \in \mathbb{C}| |z| = 1\}$. The terminology 
$L^{2}(\mathbb{R})$, or $L^{2}(\mathbb{R}^{n})$ will refer to the usual 
$L^{2}$-spaces defined from Lebesgue measure. We will have occasion to 
consider other measures as well, but it will be understood that they are 
Borel measures. They may be singular or not. The terminology ``absolutely 
continuous" will refer to the comparison with Lebesgue measure in the 
appropriate context. 

      To see the connection to spectral multiplicity for unitary operators, 
consider the unitary operator $T$ in the Hilbert space 
$\mathcal{H}:= L^{2}(\mathbb{R})$ of translation by $1$, i.e., 
$(Tf)(x) := f(x - 1)$. In this case, this unitary operator $T$ is 
represented by multiplication via the Fourier transform $W$, where we view 
$W$ as a unitary operator, $W : L^{2}(\mathbb{R}) \to L^{2}(\mathbb{R})$. 
But we shall be interested in localizing the analysis of $T$, so in 
considering the action of $T$ on a single vector $\psi$ in 
$L^{2}(\mathbb{R})$, or on a finite family. The action of $T$ on a single 
vector is determined by a scalar measure, and the action on a finite family 
of vectors by a matrix valued measure, as we outline below.   

     In order to make a direct connection to the study of closed translation 
invariant subspaces in $L^{2}(\mathbb{R})$, or $L^{2}(\mathbb{R}^{d})$ we 
begin in section \ref{sec:2} with the case when our Hilbert space 
$\mathcal{H}$ is $L^{2}(\mathbb{R})$, and when the unitary operator $T$ 
is translation by $1$, i.e., $(Tf)(x) := f(x - 1)$.
In that case, of course, the powers of $T$ by $k \in \mathbb{Z}$, 
$(T^{k} f)(x) := f(x - k)$  for $k \in \mathbb{Z}$, $x \in \mathbb{R}$, and 
$f \in L^{2}(\mathbb{R})$.

       In the context of over completeness in the Hilbert space 
$L^{2}(\mathbb{R})$, the question of non-trivial infinite linear 
combinations arises naturally, for example in the study of translation 
invariant subspaces in $L^{2}(\mathbb{R})$, see \cite{Dau92} and \cite{Hel64}. 
In this context, we must study infinite sums 
$\sum_{k \in \mathbb{Z}} c_{k}T^{k}f$ of the translates 
$(T^{k} f)(x) := f(x - k)$ for $k \in \mathbb{Z}$, when $f$ is a fixed 
function. Since these translates are not assumed to be orthogonal, one must 
be careful in the consideration of such infinite combinations. For example 
should the coefficients $c_{k}$, $k \in \mathbb{Z}$, be chosen from $l^{2}$? 
Or from what sequence space? 

      The question of infinite linear dependencies of translates arises in 
Kolmogorov's and Wiener's prediction theory (see e.g., \cite{MiSa80}), and 
it is referred to there as $L^{2}$-independence vs $L^{2}$-dependence. This 
framework will be naturally included in our considerations below. However, 
we will begin with $L^{2}(\mathbb{R})$-considerations. 

      As it turns out, our approach applies more generally to families 
of $n$ arbitrary commuting unitary operators in a complex Hilbert space 
$\mathcal{H}$, or equivalently to the spectral theory of a unitary 
representation $U$ of the rank-$n$ lattice $\mathbb{Z}^{n}$ in 
$\mathbb{R}^{n}$. Starting with a non-zero vector $\psi$ in $\mathcal{H}$, 
we look for relations among the vectors in the cyclic subspace 
$\mathcal{H}(\psi)$ in $\mathcal{H}$ generated by $\psi$. Since these 
vectors 
\[
  \{U(k)\psi | k \in \mathbb{Z}^{n}\}
\]
involve infinite ``linear combinations," the problem arises of giving 
geometric characterizations of these non-trivial linear relations.
 
      This is the setup in section \ref{sec:3} below. Our study of the 
cyclic subspace $\mathcal{H}(\psi)$ is done with the aid of an 
associated spectral measure $\mu = \mu_{\psi}$, and we begin in 
section \ref{sec:3} by introducing an isometric isomorphism from 
$L^{2}(\mu)$ onto the cyclic subspace $\mathcal{H}(\psi)$.
 
     In the multivarable case, we must make use of funcion theory on the 
$n$-torus $\mathbb{T}^{n}$, and the reader is referred to \cite{Rud69, Rud86} 
for that.

\section{A Lemma}
\label{sec:1a}
In wavelet theory, there are natural choices of Hilbert spaces and of 
generating functions.  For each such choice, one wishes to consider linear 
combinations for the purpose of decomposing general vectors.  For this to 
be successful it helps that the vectors $f_{k}$ have the following 
representation as outlined in Introduction: with the use of 
a single unitary operator $U$ in $\mathcal{H}$, if there is a 
fixed vector $\varphi \in \mathcal{H}$ such that $f_{k} = U^{k}\varphi$ 
for all $k \in \mathbb{Z}$. 

The following lemma gives a necessary and sufficient condition for 
this to work.  The crux is two versions of stationarity.

\begin{lem}
\label{L:1.1}
Let $\mathcal{H}$ be a Hilbert space, and let 
$\{f_{k}\}_{k \in \mathbb{Z}} \subset \mathcal{H} \setminus \{0\}$ be 
given.  Then the following two conditions are equivalent:
\begin{itemize}
  \item[(i)] $\langle f_{j+n} | f_{k+n} \rangle_{\mathcal{H}}
  = \langle f_{j} | f_{k} \rangle_{\mathcal{H}}$,  for all 
  $j, k, n, \in \mathbb{Z}$; and\\
  \item[(ii)] There is a Gaussian probability space $(\Omega, P)$, and a 
  stationary stochastic process $(X_{k})_{k \in \mathbb{Z}}$, i.e., each 
  $X_{k}: \Omega \to \mathbb{C}$ a random variable, and an isometry 
  $W: \mathcal{H} \to L^{2}(\Omega, P)$ such that $Wf_{k}=X_{k}$ for all 
  $k \in \mathbb{Z}$, and 
  \[
    \langle f_{j} | f_{k} \rangle_{\mathcal{H}} = E(\overline{X_{j}} X_{k})
   = \int_{\Omega}\overline{X_{j}}(\omega)X_{k}(\omega)dP(\omega).
  \]
\end{itemize}
\end{lem}
\begin{proof}
For every $n \in \mathbb{N}$, consider the usual Gaussian density in 
$\mathbb{C}^{2n+1}$, i.e., $z=(z_{-n}, \cdots, z_{0}, z_{1}, \cdots, z_{n})$ 
with mean $E_{n}(z_{j})=0$ and covariance 
\[
  E_{n}(\overline{z_{j}}z_{k})=\langle f_{j} | f_{k} \rangle_{\mathcal{H}} 
  \text{  , } -n \leq j, k \leq n.
\]
It is easy to see that this is a Kolgomorov consistent system 
\cite{Jor06, Nel69}.  The existence of the infinite Gaussian space 
$(\Omega, P)$ with the desired properties now follows.  In particular $\Omega$ 
is a space of function $\omega : \mathbb{Z} \to \mathbb{C}$, and 
$X_{k}(\omega) = \omega(k)$, $k \in \mathbb{Z}$.
\end{proof}

\begin{rem}
\label{R:1.2}
In section \ref{sec:5} we will be using an analogous family of Gaussian 
Hilbert spaces, and the associated spectral density functions.  
\end{rem}

\section{Notation and Statement of the Problem}
\label{sec:2}
\subsection*{Preliminaries}
While we will be stating our main results in the general context of 
unitary operators in Hilbert space (section \ref{sec:3} below), and 
stochastic processes (section \ref{sec:5}), it is helpful to first consider 
a prime example illustrating the kind of interplay between function theory 
and Hilbert space geometry. Such an example is afforded by the case of the 
unit-translation operator in the Hilbert space $L^{2}(\mathbb{R})$, and we 
begin with this below.

We will study the interconnection between $L^{2}(\mathbb{R})$-functions 
and their restrictions to unit intervals $I_{n}=[n,n+1)$ for 
$n \in \mathbb{Z}$.  Since the individual intervals $I_{n}$ arise from
$I_{0}=[0,1)$ by integer translation, it will be convenient to compare
the two Hilbert spaces $L^{2}(0,1)$ and $L^{2}(\mathbb{R})$.  

The $L^{2}$-norm on $\mathbb{R}$ may be decomposed as 
\[
  \int_{\mathbb{R}}|\psi(x)|^{2}dx 
  = \sum_{n \in \mathbb{Z}} \int_{n}^{n+1} |\psi(x)|^{2}dx
  = \int_{0}^{1} \sum_{n \in \mathbb{Z}} |\psi(x+n)|^{2}dx.
\]
So if we introduce $p_{\psi}(x):=\sum_{n\in\mathbb{Z}}|\psi(x+n)|^{2}
= \underset{\mathbb{Z}}{\text{PER}}|\psi|^{2}$, then $p_{\psi} \in L^{1}(0,1)$, and 
$\int_{0}^{1}p_{\psi}(x)dx=\int_{\mathbb{R}}|\psi(x)|^{2}dx
=\|\psi\|_{L^{2}(\mathbb{R})}^{2}$

We will study what properties of a given $L^{2}(\mathbb{R})$-function $\psi$
can be predicted from the local p-version.  The problem is intriguing, since
the fixed function $p \in L^{1}(0,1)$ may be decomposed in many different 
ways as $p(x)=\text{PER}|\psi(x)|^{2}$ with $\psi \in L^{2}(\mathbb{R})$.

We will study the initial function $p$ via $L^{2}(\mu)$ where 
\begin{equation}
\label{E:mu}
  d\mu(x)=p(x)dx.
\end{equation}
In fact, we will study $L^{2}(\mu)$ even if the measure $\mu$ is not assumed
to be absolutely continuous with respect to Lesbegue measure.  But in the
absolutely continuous case, the function $p$ will pop up as a Radon-Nikodym
derivative.  

Moreover, we will need the following matrix version of (\ref{E:mu}).  Let 
$N$ be a positive integer and let $M_{N}^{+}$ denote the complex $N \times N$
matrices $P$ satisfying $spec(P) \subset [0, \infty)$.

\begin{defn}
We say that a measurable function $[0,1) \ni x \mapsto P(x) \in M_{N}^{+}$ is
$L^{1}$ if for all vectors $v \in \mathbb{C}^{N}$ the functions 
$x \to \langle v | P(x) v \rangle$ is in $L^{1}(0,1)$
\end{defn}

\subsubsection*{Notation} For $v, w \in \mathbb{C}^{N}$, we set 
\[
  \langle v | w \rangle := \sum_{k=1}^{N} \overline{v_{k}}w_{k}.
\]
If $\mu$ is any Borel measure supported on a subset of $\mathbb{R}$, and 
$f, g \in L^{2}(\mu)$, we set 
\[
  \langle f | g \rangle_{L^{2}(\mu)}= \int \overline{f(x)}g(x)d\mu (x).
\]

Let $\mathcal{F} \subset L^{2}(\mathbb{R})$ be a finite subset of $N$ 
elements.  We then define a function 
\[
  P=P_{\mathcal{F}}:[0,1] \to M_{N}^{+}
\]
as follows.  Let $\mathcal{F}$ index the rows and the columns in $P$ via
$(P_{\varphi, \psi})$ $\varphi, \psi \in \mathcal{F}$, where 
\[
  (P_{\varphi, \psi})=\text{PER}(\overline{\varphi(x)}\psi(x)).
\]
An easy calculation shows that $P$ is $L^{1}$, and that 
\[
  \int_{0}^{1} P_{\varphi, \psi}(x)dx 
  = \langle \varphi | \psi \rangle_{L^{2}(\mathbb{R})}
  \text{, for all } \varphi, \psi \in \mathcal{F}.
\]

\begin{lem}
\label{L:2.2}
Let $P:[0,1] \to M_{N}^{+}$ be given and suppose that 
$x \mapsto \|P(x)\|_{2,2} \in L^{\infty}$ then for all 
$v \in \mathbb{C}^{N}$, we have the estimate 
\[
  \|P(x)v\|^{2} \leq \lambda \langle v | P(x)v \rangle
\] where $\lambda := \underset{x}{\text{sup}} \|P(x)\|_{2,2}$
\end{lem}
\begin{proof}
Since $ \|P(x)v\|^{2} = \langle v | P(x)^{2}v \rangle$, the assertion
is a statement about a single $P \in M_{N}^{+}$ refering now to the usual
ordering on the Hermitian matrices.  Let $P=\sum \lambda_{i}E_{i}$ be the
spectral resolution with $(E_{i})$ denoting orthogonal projections,
i.e., $E_{j}E_{k}=\delta_{j,k}E_{j}$.  Then 
\[
  P^{2}=\sum_{j}\lambda_{j}^{2}E_{j} \leq 
  \underset{k}{\text{max}} \lambda_{k} \sum_{j} \lambda_{j}E_{j} 
  = (\underset{k}{\text{max}} \lambda_{k}) \text{ } P.
\]
The result is immediate from this.
\end{proof}

\begin{defn}
\label{D:2.3}
We shall use the following notation $e_{k}$, $k \in \mathbb{Z}$ for the Fourier
basis in $L^{2}(0,1)$:
\[
  e_{k}(x) = e^{i2\pi kx}, \text{  } x \in [0,1].
\]
With this convention, the Parseval identity in $L^{2}(0,1)$ reads, 
\begin{equation}
\label{E:ckl}
  \sum_{k \in \mathbb{Z}} |c_{k}|^{2} 
  = \int_{0}^{1}\left\vert \sum_{k \in \mathbb{Z}}c_{k}e_{k}(x)\right\vert^{2}dx, \text{ for all }
  c=(c_{k}) \in l^{2}(\mathbb{Z}).
\end{equation}
\end{defn}

\begin{defn}
\label{D:2.4}
Let $\mu$ be a positive Borel measure supported in the unit-interval 
$[0,1] \subset \mathbb{R}$.  Assume $\mu([0,1])< \infty$.  Let 
$\mathcal{D} \subset l^{2}(\mathbb{Z})$ of all finite sequence 
$(c_{k})_{k \in \mathbb{Z}}$; i. e., $c_{k} = 0$ for all but a finite set
of index values of $k$.  We define a linear operator 
$F: \mathcal{D} \to L^{2}(\mu)$ as follows: 
\begin{equation}
\label{E:Fck}
  F((c_{k})):= \sum_{k}c_{k}e_{k}(x).
\end{equation}
\end{defn}

\begin{lem}
\label{L:2.5}
If 
\begin{equation}
\label{E:ek}
  \sum_{k \in \mathbb{Z}}\left\vert \int e_{k}(x)d\mu (x)\right\vert^{2} < \infty,
\end{equation} 
then the operator $F$ is closable.
\end{lem}
\begin{proof}
An operator is said to be \textit{closable} if the closure of its graph 
$(c, Fc)_{c \in \mathcal{D}}$ is again a graph of a linear operator.  In 
that case, we say that the resulting operator is the closure of $F$.  It 
is known that to test closability we only need to check that the domain of 
the adjoint operator $F^{*}$ is dense.  The formula for $F^{*}$ is as follows:
\[
  \langle F^{*}f | c \rangle_{l^{2}} = \langle f | Fc \rangle_{L^{2}(\mu)}, 
  \text{  } f \in \text{dom}(F^{*}), \text{  } c \in \mathcal{D}.
\]
It follows that 
\[
  (F^{*}f)_{k} = \int_{0}^{1}\overline{e_{k}(x)}f(x)d\mu (x)
\]
Hence, (\ref{E:ek})
is simply saying that the Fourier functions $e_{k}$ belong to dom$(F^{*})$.
Since the span of the functions $e_{k}$ is dense in $L^{2}(\mu)$, we 
conclude that $F$ is closable.
\end{proof}

We shall also need the following fundamental result from operator theory:
\begin{cor}
Suppose (\ref{E:ek}) is satisfied, denote the closure of the operator $F$
by the same symbol.  Then $F^{*}F$ is selfadjoint operator, and there is a 
partial isometry $U$ such that $F=U(F^{*}F)^{1/2}=(FF^{*})^{1/2}U$.  The 
partial isometry $U$ will be chosen with initial space equal to 
$l^{2} \ominus ker(F)$ and final space equal to the closure of the range 
of $F$. 
\end{cor}

\section{Unitary Operators}
\label{sec:3}

Let $\mathcal{H}$ be a (complex) Hilbert space and let $(T_{1}, ..., T_{n})$
be a finite family of commuting unitary operators in $\mathcal{H}$.  We
introduce the following multi-index notation 
$k=(k_{1}, ..., k_{n}) \in \mathbb{Z}^{n}$, i.e., $k_{j} \in \mathbb{Z}$ for
$1 \leq j \leq n$, and 
\begin{equation}
\label{E:T}
  T^{k} := T_{1}^{k_{1}}T_{2}^{k_{2}} \cdots T_{n}^{k_{n}}
\end{equation}
For vectors $\psi \in \mathcal{H} \setminus \{0\}$, we set 
\begin{equation}
\label{E:psi}
  \psi_{k} := T^{k}\psi, \text{  }k \in \mathbb{Z}^{d}
\end{equation}

The closed subspace in $\mathcal{H}$ generated by the vectors 
$\{\psi_{k} | k \in \mathbb{Z}^{d}\}$ will be denoted $\mathcal{H}(\psi)$;
and it is called the cyclic subspace generated by the vector $\psi$.

For points $z=(z_{1}, z_{2}, ..., z_{n}) \in \mathbb{T}^{n}$, we shall be
using the 
$\underbrace{[0,1) \times \cdots \times [0,1)}_{n \text{  } times} = [0,1)^{n}$ 
parametrization 
\begin{equation}
\label{E:zpara}
  z=(e^{i2\pi x_{1}}, e^{i2\pi x_{2}}, ..., e^{i2\pi x_{k}}),
\end{equation}  
and identifications:
\[
  z^{k} \longleftrightarrow (k_{1}x_{1}, k_{2}x_{2}, ..., k_{n}x_{n}) 
  \text{ if } k\in \mathbb{Z}^{n};
\]
and 
\begin{equation}
\label{E:zw}
  zw \longleftrightarrow (x_{1}+y_{1}, x_{2}+y_{2}, ..., x_{n}+y_{n})
\end{equation}  
where the addition in the coordinates in (\ref{E:zw}) are addition mod $1$,
i.e., addition in the group $\mathbb{R}/\mathbb{Z}$.  Hence 
$\mathbb{T}^{n} \simeq \mathbb{R}^{n}/\mathbb{Z}^{n}$.

Our present approach to the Spectral Representation Theorem for families 
of commuting unitary operators in Hilbert space is closest to that of 
\cite{Nel69}, a set of Lecture Notes by Ed Nelson; now out of print but 
available on URL http://www.math.princeton.edu/~nelson/. The 
\textit{multiplicity function} is presented there as a complete invariant. 
Of course, in wavelet applications, there are also the additional consistency 
relations (see \cite{BMM99}), but the notion of a multiplicity function is 
general.
 
This representation theoretic approach, adapted below to our present 
applications,  fits best with how we use the Spectral Representation 
Theorem in understanding wavelets and generalized multiresolution 
analysis (GMRAs); see also \cite{Bag00, BJMP05, BMM99}.  The standard 
presentation of the Spectral Theorem in textbooks is typically different 
from the Spectral Representation Theorem, and here we spell out the 
connection; developing here a computational approach. 

\begin{lem}
\label{L:3.1}
Let $\mathcal{H}$, $T=(T_{1}, ..., T_{n})$ and $\psi$ be as described above,
and let $\mathcal{H}(\psi)$ be the cyclic subspace in $\mathcal{H}$ 
generated by $\{\psi_{k}|k \in \mathbb{Z}^{n}\}$.  Then there is a unique 
Borel measure $\mu$ on $[0,1]^{n}$ and an isometric isomorphism
\begin{equation}
\label{E:isom}
  W:L^{2}([0,1]^{n}, \mu) \to \mathcal{H}(\psi)
\end{equation}
determined by 
\begin{equation}
\label{E:czT}
  \sum_{k \in \mathbb{Z}^{n}}c_{k}z^{k} \mapsto 
  \sum_{k \in \mathbb{Z}^{n}}c_{k}T^{k}\psi
\end{equation}
on the trigonometric polynomials.
\end{lem}

\begin{rem}
The significant part of the lemma relates to the Spectral Theorem: It is the
extension of the mapping in (\ref{E:czT}) from the trigonometric polynomials 
to the algebra of the measurable functions. 
\end{rem}

\begin{proof}
(of Lemma \ref{L:3.1})We first elaborate the formula (\ref{E:czT}).  Set
\begin{align*}
  z^{k} :&= e_{k}(x)  \\
        &= e^{i2\pi k_{1}x_{1}}e^{i2\pi k_{2}x_{2}} \cdots e^{i2\pi k_{n}x_{n}} \\
        &=e^{i2\pi k\cdot x}
\end{align*}
with $k \cdot x:=k_{1}x_{1}+ \cdots + k_{n}x_{n}$, and for finite 
summations:
\begin{equation}
\label{E:trigrep}
  \sum_{k \in \mathbb{Z}}c_{k}z^{k} = \sum_{k \in \mathbb{Z}}c_{k}e_{k}(x).
\end{equation}
This is the representation of the trigonometric polynomials as 
$\mathbb{Z}^{n}$-periodic functions.  Set
\[
  m_{c}(x) := \sum_{k \in \mathbb{Z}^{n}}c_{k}e_{k}(x)
\]
and
\begin{equation}
\label{E:mcTpsi}
  m_{c}(T)\psi := \sum_{k \in \mathbb{Z}^{n}}c_{k}T^{k}\psi.
\end{equation}
By the Spectral Theorem, this is a projection valued measure on $[0,1]^{n}$
such that
\begin{equation}
\label{E:mcE}
  m_{c}(T) = \int_{[0,1]^{n}}m_{c}(x)E(dx).
\end{equation}

The measure $E$ is defined on the Borel sets $\mathcal{B}_{n}$ in $[0,1]^{n}$:
\[
  E:\mathcal{B}_{n} \to \text{PROJ}(\mathcal{H})
\]
and 
\begin{equation}
\label{E:EBorel}
  E(A_{1} \cap A_{2}) = E(A_{1})E(A_{2}), \text{ for all }A_{1}, A_{2} \in 
  \mathcal{B}_{n}.
\end{equation}
A linear operator $E:\mathcal{H} \to \mathcal{H}$ is said to be a projection
if and only if $E=E^{*}=E^{2}$.  Note that condition (\ref{E:EBorel}) for
a projection valued measure implies that for $A_{1} \cap A_{2} = \emptyset$,
the subpsaces $E(A_{1})\mathcal{H}$ and $E(A_{2})\mathcal{H}$ are 
orthogonal.
The measure $\mu=\mu_{\psi}$ in the conclusion in the lemma is 
\begin{align*}
  \mu(A) :&= \|E(A)\psi\|^{2} \\
         &= \langle \psi | E(A)\psi \rangle_{\mathcal{H}}.
\end{align*}
Since 
\begin{align*}
\label{E:kUT}
  \mathbb{Z}^{n} \ni k \mapsto U(k) :&= T^{k} \\
                                    &= T_{1}^{k_{1}}T_{2}^{k_{2}} \cdots 
                                       T_{n}^{k_{n}}
\end{align*}
is a unitary representation, there is a projection valued measure
\[
  \mathcal{B}(\mathbb{T}^{n}) \ni A \mapsto E(A) \in \text{PROJ}(\mathcal{H}) 
\]
such that
\begin{align*}
  I_{\mathcal{H}} &= \int_{\mathbb{T}^{n}}E(dx), \text{ and} \\
             U(k) &= \int_{\mathbb{T}^{n}}e_{k}(x)E(dx) \text{, for all } 
  k \in \mathbb{Z}^{n}.
\end{align*}
This means that for every measurable function 
$m:\mathbb{T}^{n} \to \mathbb{C}$, the operator $m(T)$ may be defined by the
functional calculus
\begin{equation}
\label{E:fuctcalc}
  m(T) = \int_{\mathbb{T}^{n}}m(x)E(dx).
\end{equation}
Setting $\mu=\mu_{\psi}$ we conclude that $\mu$ is a scalar Borel measure,
i.e., 
\begin{align}
\label{E:mupsiE}
  \mu_{\psi}(A) :&= \langle \psi | E(A)\psi \rangle_{\mathcal{H}} \\
              &=\|E(A)\psi\|_{\mathcal{H}}^{2}
\end{align}
We have used that
\begin{equation}
\label{E:E(A)}
  E(A)=E(A)^{*}=E(A)^{2} \text{, for all } A \in \mathcal{B}(\mathbb{T}^{n}).
\end{equation}

\begin{lem}
\label{L:3.3a}
Let $Q:\mathcal{H} \to \mathcal{H}$ be a linear operator.  Then the 
following are equivalent:
\begin{itemize}
  \item[(a)] $Q U(k) = U(k)Q$, for all $k \in \mathbb{Z}^{n}$; and
  \item[(b)] $QE(A)=E(A)Q$, for all $A \in \mathcal{B}(\mathbb{T}^{n})$.
\end{itemize}
In summary, a bounded operator $Q$ commutes with the unitary representation 
$U$ if and only if it commutes with the spectral projections.
\end{lem}
\begin{proof}
Left to the reader.
\end{proof}
We are now ready to define the isometry $W=W_{\psi}$ from (\ref{E:isom}),
where
\begin{equation}
\label{E:psi1}
  W_{\psi} : L^{2}(\mathbb{T}^{n}, \mu_{\psi}) \to \mathcal{H}(\psi).
\end{equation}
Set
\begin{equation}
\label{E:psi2}
  W_{\psi}(m) := m(T)\psi
\end{equation}
with $m$ a measurable function on $\mathbb{T}^{n}$, and with $m(T)$
defined by the measurable functional calculus in (\ref{E:fuctcalc}).

The verification of the isometric property for $W_{\psi}$ is as follows:
\begin{align*}
  \|W_{\psi}(m)\|_{\mathcal{H}}^{2} &\underset{by (\ref{E:fuctcalc})}{=}
  \|m(T)\psi\|^{2} \\
  &= \left\Vert \int_{\mathbb{T}^{n}}m(x)E(dx)\psi \right\Vert_{\mathcal{H}}^{2} \\
  &= \langle \int_{\mathbb{T}^{n}}m(x)E(dx)\psi \vert \int_{\mathbb{T}^{n}}m(y)E(dy)\psi \rangle_{\mathcal{H}} \\
  &= \int_{\mathbb{T}^{n}}\int_{\mathbb{T}^{n}} \overline{m(x)}m(y)
  \langle E(dx)\psi|E(dy)\psi \rangle_{\mathcal{H}} \\
  &\underset{by (\ref{E:E(A)})}{=} \int_{\mathbb{T}^{n}}
  \int_{\mathbb{T}^{n}} \overline{m(x)}m(y)\langle \psi|E(dx)E(dy)\psi 
  \rangle_{\mathcal{H}} \\
  &\underset{by (\ref{E:EBorel})}{=} \int_{\mathbb{T}^{n}}
  |m(x)|^{2} \langle \psi|E(dx)\psi \rangle_{\mathcal{H}} \\
  &\underset{by (\ref{E:mupsiE})}{=} \int_{\mathbb{T}^{n}}|m(x)|^{2}
d\mu_{\psi}(x) \\
  &= \|m\|_{L^{2}(\mathbb{T}^{n}, \mu_{\psi})}^{2}.
\end{align*}
While the calculation is done initially for $m \in L^{\infty}(\mu_{\psi})$,
after the isometric property
\begin{equation}
\label{E:3.17}
  \|m(T)\psi\|_{\mathcal{H}}=\|m\|_{L^{2}(\mu_{\psi})}
\end{equation}
is verified, it follows that
\[
  W_{\psi}(m) := m(T)\psi
\]
is now well defined for all $m \in L^{2}(\mu_{\psi})$; and the operator
$W_{\psi}$ resulting by $L^{2}$-norm completion will be defined on all of
$L^{2}(\mu_{\psi})$.

To show that
\begin{equation}
\label{E:psiLH}
  W_{\psi}L^{2}(\mu_{\psi})= \mathcal{H}(\psi)
\end{equation}
we check that if $\xi \in \mathcal{H}(\psi)$, and 
\begin{equation}
\label{E:ximT}
  \langle \xi | m(T)\psi \rangle_{\mathcal{H}} = 0 \text{, for all } m \in 
  L^{2}(\mu_{\psi}),
\end{equation}
then $\xi =0$.

First note that for $A \in \mathcal{B}(\mathbb{T}^{n})$ we have 
\[
  |\langle \xi | E(A)\psi \rangle|^{2} \leq \|\xi\|^{2}\mu_{\psi}(A).
\]
Hence, by the Radon-Nikodym Theorem, there is a function 
$F_{\xi} \in L^{1}(\mu_{\psi})$ such that
\[
  \langle \xi | E(dx)\psi \rangle = F_{\xi}(x)d\mu_{\psi}(x),
\]
and
\begin{equation}
\label{E:intmFxi}
  \int_{\mathbb{T}}m(x)F_{\xi}(x)d\mu_{\psi}(x)=0
\end{equation}
for all functions $m$ as above.

As a result, 
\begin{equation}
\label{E:xiF}
  \xi = \overline{F}_{\xi}(T)\psi.
\end{equation}
But since (\ref{E:intmFxi}) holds for all $m$, we conclude that $F_{\xi}=0$,
$\mu_{\psi}$ a. e.

Substituting back into (\ref{E:xiF}), we conclude that $\xi=0$.  But 
$W_{\psi}$ is isometric, so its range is closed.  Since it is dense, the
desired conclusion (\ref{E:psiLH}) now follows.
\end{proof}

We now turn to the case of matrix-value measures.  The setting is as above:
$T=(T_{1}, ..., T_{n})$ a given set of commuting unitary operators acting
in a Hilbert space $\mathcal{H}$, i.e., 
\begin{equation}
\label{E:Tj}
  T_{j}:\mathcal{H} \to \mathcal{H}, \text{ } 1\leq j \leq n.
\end{equation}
The main difference is that we will be considering cyclic subspaces in 
$\mathcal{H}$ generated by a fixed finite family 
$\mathcal{F} = \{\psi_{1}, \psi_{2}, ..., \psi_{N}\}$ in 
$\mathcal{H} \setminus \{0\}$.  We let $\mathcal{H}(\mathcal{F})$ denote the 
closed span of the vectors
\begin{equation}
\label{E:setT}
  \{T^{k}\psi_{j}|k\in \mathbb{Z}^{n}, 1 \leq j \leq n\}.
\end{equation}

Let
\[
  \mathcal{B}(\mathbb{T}^{n}) \ni A \mapsto E(A) \in \text{PROJ}(\mathcal{H})
\]
be the projection valued measure introduced in (\ref{E:fuctcalc}), i.e.,
satisfying
\begin{equation}
\label{E:TintzkE}
  T^{k} = \int_{\mathbb{T}^{n}}z^{k}E(dz),
\end{equation}
or in additive notation
\begin{equation}
\label{E:TintekE}
  T^{k} = \int_{\mathbb{T}^{n}}e_{k}(x)E(dx).
\end{equation}

Setting
\begin{equation}
\label{E:Prs}
  P_{r,s}(A):= \langle \psi_{r}|E(A)\psi_{s} \rangle, \text{ for }
  1 \leq r, s \leq N, \text{  } A \in \mathcal{B}, 
\end{equation}
we note that $P(\cdot)$ is a matrix-valued measure on $\mathbb{T}^{n}$,
i.e., taking values $M_{N}^{+}$.

Specifically, let $v \in \mathbb{C}^{N}$, and 
$A \in \mathcal{B}(\mathbb{T}^{n})$.  Setting 
$\psi := \sum_{r=1}^{N}v_{r}\psi_{r}$, we get: 
\begin{align*}
  \langle v|P(A)v \rangle_{l_{N}^{2}} 
  &= \sum_{r}\sum_{s}\overline{v}_{r}\langle \psi_{r}|E(A)\psi_{s} \rangle 
  v_{s} \\
  &= \langle \sum_{r}v_{r}\psi_{r}|E(A)\sum_{s}v_{s}\psi_{s}
  \rangle_{\mathcal{H}} \\
  &= \|E(A)\psi\|_{\mathcal{H}}^{2} \geq 0,
\end{align*}
as claimed. 

\begin{defn}
\label{D:3.2}
Let $P$ be the matrix-valued measure defined on $\mathcal{B}(\mathbb{T}^{n})$
as in (\ref{E:Prs}).  The Hilbert space $L^{2}(P)$ then consists of all
measurable functions $m:\mathbb{T}^{n} \to \mathbb{C}^{N}$ such that
\begin{equation}
\label{E:mnormp}
  \|m\|_{P}^{2}:= \int_{\mathbb{T}^{n}}\langle m| P(dx)m \rangle_{l_{N}^{2}}
  < \infty
\end{equation}
where $m^{T} = (m_{1}, ..., m_{N})$, and 
\begin{equation}
\label{E:mPm}
  \langle m| P(dx)m \rangle_{l_{N}^{2}} := 
  \sum_{r}\sum_{s} \overline{m_{r}(x)}P_{r,s}(dx)m_{s}(x).
\end{equation}
\end{defn}

\begin{lem}
\label{L:3.3}
Let $\mathcal{H}$, $T=(T_{1}, ..., T_{n})$ and 
$\mathcal{F} = \{\psi_{1}, \psi_{2}, ..., \psi_{N}\}$ be as described above.
Let $P=P_{\mathcal{F}}$ be the corresponding matrix-value measure, and let
$\mathcal{H}(\mathcal{F})$ denote the (closed) cyclic subspace in 
$\mathcal{H}$ generated by 
$\{T^{k}\psi|k \in \mathbb{Z}^{n}, \psi \in \mathcal{F} \}$.  For 
$m \in L^{2}(P)$, set
\begin{equation}
\label{E:Wmsum}
  W(m) := \sum_{r=1}^{N} m_{r}(T)\psi_{r}.
\end{equation}
Then $W$ defines an isometric isomorphism
\[
  W:L^{2}(P) \to \mathcal{H}(\mathcal{F})
\]
mapping $L^{2}(P)$ onto $\mathcal{H}$.
\end{lem}
\begin{proof}
Except for technical modification, the proof of the lemma follows the idea
in the proof of Lemma \ref{L:3.1}, which is the special case of $N=1$.

Hence we restrict our present discussion to the verification that $W$ is 
isometric from $L^{2}(P)$ into $\mathcal{H}$.

For $m$ in $L^{2}(P)$, set $m^{T}:=(m_{1}, ..., m_{r})$, the scalar coordinate
functions.  Then
\begin{align*}
  \|Wm\|_{\mathcal{H}}^{2} 
  &= \left\Vert \sum_{r=1}^{N}m_{r}(T)\psi_{r}\right\Vert_{\mathcal{H}}^{2} \\
  &= \left\Vert \sum_{r=1}^{N} \int_{\mathbb{T}^{n}}m_{r}(x)E(dx)\psi_{r}\right\Vert_{\mathcal{H}}^{2} \\
  &= \sum_{r=1}^{N}\sum_{s=1}^{N} \int_{\mathbb{T}^{n}} \overline{m_{r}(x)}
  P_{r,s}(dx)m_{s}(x) \\
  &= \int_{\mathbb{T}^{n}} \langle m| P(dx)m \rangle \\
  &\underset{by \text{ } (\ref{E:mPm})}{=} \|m\|_{L^{2}(P)}^{2}.
\end{align*}
The arguments from Lemma \ref{L:3.1} shows that $W$ in fact maps onto 
$\mathcal{H}(\mathcal{F})$.
\end{proof}

\section{Corollaries and Applications}
\label{sec:4}

In this section we return to the main application of the general 
spectral theory developed in section \ref{sec:3}.

While the setting in section \ref{sec:3} applies, in the single variable 
case, to a general unitary operator $T$ in Hilbert space $\mathcal{H}$; 
and in the multivariable case to a finite commuting family 
$(T_{1}, T_{2}, ..., T_{n})$ of unitary operators, the main application 
is to $\mathcal{H} = L^{2}(\mathbb{R})$, or to 
$\mathcal{H} = L^{2}(\mathbb{R}^{n})$.

In the single variable case, the unitary operator $T$ will be translation 
by $1$ to the right of functions on the real line $\mathbb{R}$; and in 
the multivariable case, the system will consist of these translation 
operators but now referring to unit-translation in the $n$ coordinate 
directions.

Our first Corollary of the general formulas for the spectral 
measure/function in section \ref{sec:3} will show that, in the 
translation case, the spectral measure $\mu$ is absolutely continuous 
with respect to Lebesgue measure, and that the Radon-Nikodym derivative 
is the function introduced above in section \ref{sec:2}. 

While the results apply to the general multivariable setting, for clarity 
we will only give full details in the single variable case. But with the aid 
of section \ref{sec:2}, the reader will be able to extend the formulas 
from $n = 1$ to $n > 1$.

We stress that there is a spectral measure $\mu = \mu_{\psi}$ for every 
vector in the Hilbert space. But this section is concerned with the 
Hilbert spaces $\mathcal{H} = L^{2}(\mathbb{R}^{n})$. And so the spectral 
measure $\mu = \mu_{\psi}$ depends on which function $\psi$ is chosen in 
$\mathcal{H} = L^{2}(\mathbb{R}^{n})$.

\begin{cor}
\label{C:Tf}
Let $(Tf)(x):=f(x-1)$ be the translation in $L^{2}(\mathbb{R})$, and let
\begin{equation}
\label{E:fFF}
  \widehat{f}(t):=\int_{\mathbb{R}}e^{-i2\pi tx}f(x)dx
\end{equation}
be the $L^{2}$-Fourier transform.
Let $\psi \in L^{2}(\mathbb{R}) \setminus \{0\}$, and let $\mu := \mu_{\psi}$
be the spectral measure introduced in Lemma \ref{L:3.1}.  Then
\begin{itemize}
  \item[(i)] $\mu$ is absolutely continuous with respect to Lebesgue 
    measure $dt$ on $[0,1]$;
  \item[(ii)] The Radon-Nikodym derivative is as follows
    \begin{equation}
    \label{E:RNder}
      \frac{d\mu_{\psi}}{dt} = \sum_{n \in \mathbb{Z}}|\widehat{\psi}(t+n)|^{2}
      = \text{PER}|\widehat{\psi}|^{2}(t);
    \end{equation}
  and, in particular,
  \item[(iii)] $\text{PER}|\widehat{\psi}|^{2}$ is in $L^{1}(0,1)$. 
\end{itemize}
\end{cor}
\begin{proof}
Recall the functional calculus defined for the Borel functions $m$ on 
$[0,1) \simeq \mathbb{R}/\mathbb{Z}$ introduced in Lemma \ref{L:3.3} by
extending the following formula
\begin{equation}
\label{E:pser}
  p(z)=\sum_{k}c_{k}z^{k} \mapsto \sum_{k}c_{k}T^{k}
\end{equation}
defined initially on the polynomials where the following notation and 
identification is used: $p(t):=p(e^{i2\pi t})$, and 
\[
  (T^{k}f)(x) = \underbrace{T \circ \cdots \circ T}_{k \text{ } times}f(x)
              = f(x-k) \text{ for } x \in \mathbb{R}, \text{ and }
              k \in \mathbb{Z}.
\]

If $m$ is a Borel function on $[0,1)$ we let $m(T)$ denote the Borel 
functional calculus which arises from the extension of (\ref{E:pser})
via Lemma \ref{L:3.1}: If $A \in \mathcal{B}([0,1])$, we will use this
for the function $m:=\chi_{A}$.  In particular,  $\chi_{A}$ will 
automatically be extended from $[0,1]$ to $\mathbb{R}$ by periodicity.

As a result 
\begin{align*}
  \int_{A} \text{PER}(|\widehat{\psi}|^{2})dt 
  &= \int_{0}^{1}\chi_{A}(t) \text{PER}|\widehat{\psi}|^{2}(t)dt \\
  &= \int_{-\infty}^{\infty}\chi_{A}(t) |\widehat{\psi}(t)|^{2}dt \\
  &= \|\chi_{A}(T)\psi\|^{2} \\
  &= \mu_{\psi}(A).  
\end{align*}
This proves that $d\mu_{\psi}$ is absolutely continuous with respect to 
$dt$ on $[0,1]$, and that the formula (\ref{E:RNder}) holds for the 
Radon-Nikodym derivative $\frac{d\mu_{\psi}}{dt}$.  Since 
$\frac{d\mu_{\psi}}{dt} \in L^{1}(0,1)$ by the Radon-Nikodym theorem, we 
conclude that the almost everywhere defined function 
$\text{PER}|\widehat{\psi}|^{2}$ is well defined, and in $L^{1}(0,1)$.
\end{proof}

We now turn to several generalizations, including the multivariable case, 
and some cases of unitary operators $T$ different from the translation 
operator in $L^{2}(\mathbb{R})$. There is an additional question which is 
motivated by applications, and which we will answer concern conditions for 
when the Radon-Nikodym derivative of the spectral measure 
$\mu = \mu_{\psi}$ is in $L^{2}$, and in $L^{\infty}$. 

We now turn to a variety of structural properties that may hold for a 
bilateral sequence of vectors obtained by the application of powers of a 
single unitary operator to a fixed non-zero vector $\psi$ in a Hilbert 
space. These are conditions which arise in harmonic analysis \cite{Chr03}, 
in wavelets \cite{BMM99, Bag00, BJMP05}, and in signal processing 
\cite{CKS06}, and they are denoted by the name ``frame." But there is a 
host of distinct frame conditions, and our next result shows that they are 
determined by specific properties of our associated measure 
$\mu_{\psi}$. 

In addition, we show that under certain conditions, the vector $\psi$ may 
be replaced by a renormalized version which has the effect of turning the 
spectral Radon-Nikodym derivative into the indicator function for a certain 
Borel set.

\begin{defn}
\label{D:4.2}
Let $T:\mathcal{H} \to \mathcal{H}$ be a unitary operator, and let 
$\psi \in \mathcal{H}\setminus \{0\}$.  For $k \in \mathbb{Z}$, set 
\begin{equation}
\label{E:psiT}
  \psi_{k} := T^{k}\psi.
\end{equation}
Let $\mathcal{H}(\psi)$ be the \textit{cyclic subspace} generated by $\psi$, 
i.e., the closed space of the family $\{\psi_{k}|k \in \mathbb{Z}\}$; 
see Lemma \ref{L:3.1}.

We shall be concerned with the following properties for the sequence of 
vectors $\psi_{k}$, $k \in \mathbb{Z}$:
\begin{itemize}
  \item ONB:  
    \begin{equation}
    \label{E:4.5} 
      \langle \psi_{j} | \psi_{k} \rangle_{\mathcal{H}} = 
      \delta_{j,k}, \text{ for all } j, k \in \mathbb{Z}; 
    \end{equation}
    i.e., $\{\psi_{j}\}$ is an orthonormal basis in $\mathcal{H}(\psi)$.
  \item \textit{Parseval}:
    \begin{equation}
    \label{E:4.6} 
      \sum_{j \in \mathbb{Z}} |\langle \psi_{j}|f \rangle|^{2} = 
      \|f\|_{\mathcal{H}}^{2}, \text{ for all } \mathcal{H}(\psi);
    \end{equation}
    i.e., $\{\psi_{j}\}$ is a Parseval frame in $\mathcal{H}(\psi)$.
  \item \textit{Bessel}: \\
    There exists $B \in \mathbb{R}_{+}$ such that 
    \begin{equation}
    \label{E:4.7} 
      \sum_{j \in \mathbb{Z}} |\langle \psi_{j}|f \rangle|^{2} \leq
      B\|f\|_{\mathcal{H}}^{2}, \text{ for all } f \in \mathcal{H}(\psi);  
    \end{equation}
    i.e., $\{\psi_{j}\}$ is a Bessel frame in $\mathcal{H}(\psi)$.
  \item \textit{Frame}: \\
    There exists $A, B \in \mathbb{R}_{+}$, $A \leq B < \infty$ such that 
    \begin{equation}
    \label{E:4.8} 
      A\|f\|_{\mathcal{H}}^{2} \leq 
      \sum_{j \in \mathbb{Z}} |\langle \psi_{j}|f \rangle|^{2} \leq
      B\|f\|_{\mathcal{H}}^{2}, \text{ for all } f \in \mathcal{H}(\psi); 
    \end{equation}
    i.e., $\{\psi_{j}\}$ is a frame in $\mathcal{H}(\psi)$.
  \item \textit{Riesz}:
    There exists $A, B \in \mathbb{R}_{+}$, $A \leq B < \infty$ such that 
    \begin{equation}
    \label{E:4.9} 
      A\|c\|_{l^{2}}^{2} \leq 
      \left\Vert \sum_{k \in \mathbb{Z}} c_{k}\psi_{k}\right\Vert_{\mathcal{H}}^{2} \leq
      B\|c\|_{l^{2}}^{2}, \text{ for all } c=(c_{j}) \in l^{2}(\mathbb{Z}); 
    \end{equation}
    i.e., $\{\psi_{j}\}$ is a Riesz basis (with bounds $(A,B)$) for 
    $\mathcal{H}(\psi)$.  A part of the definition is the convergence of 
    $\sum_{k \in \mathbb{Z}} c_{k}\psi_{k}$ in $\mathcal{H}$.
  \item \textit{Absolute Continuity}:
    We say that the measure $\mu_{\psi}$ is absolutely continuous if 
    Radon-Nikodym derivative
    \begin{equation}
    \label{E:4.11}
      \frac{d\mu_{\psi}}{dx}=p_{\psi} \in L^{1}(0,1)
    \end{equation}
    exists, where by $dx$, we mean the restriction to $[0,1]$ of Lebesgue
    measure.  It will be convenient to use the isometric isomorphism
    $L^{2}(0,1) \to L^{2}(\mathbb{T})$ defined by: 
    $f(x) \longleftrightarrow F(e_{1}(x))$ where $e_{1}(x) := e^{i2\pi x}$.
\end{itemize}
\end{defn}

The next result yields a gradation of the conditions in Definition 
\ref{D:4.2}.  In its statement we will make use of the notation 
\textit{esssupp} for essential support, and \textit{esssup} for essential 
supremum.  The abbreviation a.e. will be for almost every.  

\begin{thm}
\label{T:4.3}
We have the following implications: 
(\ref{E:4.7}) Bessel $\Longrightarrow$ (\ref{E:4.11}) Absolute Continuity 
$\Longrightarrow$ \textit{Renormalization}.

By the last condition(\textit{renormalizaton}) we mean this: There is a 
vector $\psi_{\text{REN}} \in \mathcal{H}$, and a Borel subset 
$A \subset [0,1]$ such that the following three conditions (a) - (c) hold:
\begin{itemize}
  \item[(a)] $\mathcal{H}(\psi) = \mathcal{H}(\psi_{\text{REN}})$, and
  \item[(b)] $d\mu_{\psi_{\text{REN}}}(x) = \chi_{A}(x)dx$, and
  \item[(c)] $A = \text{ess supp}(p_{\psi})$.
\end{itemize}
\end{thm}
\begin{proof}
Assume (\ref{E:4.7}).  Let $f \in \mathcal{H}(\psi)$ have the form 
$f = m(T)\psi$, see (\ref{E:3.17}), as in the proof of Lemma \ref{L:3.1}.
We then have
\begin{equation}
\label{E:4.12}
  \sum_{k \in \mathbb{Z}} |\langle \psi_{k} | f \rangle|^{2}
  =\sum_{k \in \mathbb{Z}}\left\vert\int_{0}^{1}\overline{e_{k}(x)}m(x)d\mu_{\psi}(x)\right\vert^{2}.
\end{equation}
Using (\ref{E:4.7}), we conclude that the Fourier coefficients of the measure
$m(x)d\mu_{\psi}(x)$ are in $l^{2}(\mathbb{Z})$.  Since this holds for all
$m$, it follows that $d\mu_{\psi}$ has the form
\begin{equation}
\label{E:4.13}
  d\mu_{\psi}(x) = p_{\psi}(x)dx
\end{equation}
where $p_{\psi} \in L^{1}(0,1)$ is the Radon-Nikodym derivative.  This is
condition (\ref{E:4.11}) which was asserted.  

Now substituting back into (\ref{E:4.12}) and (\ref{E:4.7}), we get
\begin{equation}
\label{E:4.14}
  \int_{0}^{1}|m(x)p_{\psi}(x)|^{2}dx \leq B\int_{0}^{1}|m(x)|^{2}p_{\psi}(x)dx
\end{equation}
Rewriting this as 
\[
  \int_{0}^{1}|\sqrt{p_{\psi}}m|^{2}p_{\psi}(x)dx \leq B\int_{0}^{1}|m|^{2}p_{\psi}dx
\]
we see that (\ref{E:4.7}) is equivalent to the boundedness of the following 
multiplication operator
\begin{equation}
\label{E:4.15}
  Q_{\psi}: m \longmapsto \sqrt{p_{\psi}}m
\end{equation}
in the Hilbert space $L^{2}([0,1], p_{\psi}dx)$, i.e., to estimate
\begin{equation}
\label{E:4.16}
  \|Q_{\psi}m\|_{L^{2}(p_{\psi})}^{2} \leq  B\|m\|_{L^{2}(p_{\psi})}^{2} 
  \text{, for all } m.
\end{equation}

Hence the function $p_{\psi}$ is essentially bounded (relative to Lebesgue 
measure), and $\text{esssup}(p_{\psi}) = \|p_{\psi}\|_{\infty} \leq B$.
Let $A=A_{\psi}:=\text{esssupp}(p_{\psi})=$ the essential support of 
$p_{\psi}$.  Then the identities
\begin{equation}
\label{E:4.17}
  p_{\psi}\chi_{A} = p_{\psi}
\end{equation}
and
\begin{equation}
\label{E:4.18}
  p_{\psi}\chi_{[0,1]\setminus A} = 0
\end{equation}
hold almost everywhere on $[0,1]$.

Using Lemma \ref{L:3.1}, and setting
\begin{equation}
\label{E:4.19}
  \xi(x) := \chi_{A_{\psi}}(x)p_{\psi}(x)^{-1/2},
\end{equation}
we note that the functional calculus applied to $T$ yields a vector
\begin{equation}
\label{E:4.20}
   \psi_{\text{REN}} := \xi(T)\psi
\end{equation}
well defined in $\mathcal{H}$.

We now apply the functional calculus of $T$ to the modified vector, i.e., to
$\psi_{\text{REN}}$, and by Lemma \ref{L:3.1}, we conclude that
\begin{align*}
  \|m(T)\psi_{\text{REN}}\|_{\mathcal{H}}^{2} 
  &= \|m\xi(T)\psi\|_{\mathcal{H}}^{2}  \\
  &\underset{by (\ref{E:3.17})}{=} \int_{0}^{1}|m\xi|^{2}p_{\psi}dx \\
  &\underset{by (\ref{E:4.17})}{=} \int_{A}|m|^{2}dx.
\end{align*}
holds for all $m$.  Conclusions (b) and (c) now follow.  Conclusion (a)
amounts to a third application of Lemma \ref{L:3.1} to $\psi_{\text{REN}}$.
\end{proof}

Some of the conclusions in the next corollary are folklore, but we 
include them for completeness.  Moreover, they arise as special cases 
of our more general theorems.

\begin{cor}
\label{C:4.4}
Let $T: \mathcal{H} \to \mathcal{H}$ be a unitary operator, and let 
$\psi \in \mathcal{H} \setminus \{0\}$ be given.  Suppose the measure
$\mu_{\psi}$ is absolutely continuous, and let 
$d\mu_{\psi}(x) = p_{\psi}(x)dx$ with $p_{\psi}$ denoting the Radon-Nikodym 
derivative.

Then for the five frame properties in definition \ref{D:4.2} concerning 
the subspace $\mathcal{H}(\psi)$ we have the following characterizations:

\begin{itemize}
  \item ONB (\ref{E:4.5}): $p_{\psi} \equiv 1$, i.e., $p_{\psi}(x) = 1$
  for almost every $x \in [0,1]$. \\
  \item \textit{Parseval} (\ref{E:4.6}): There exists 
  $S \in \mathcal{B}([0,1])$, $|S|>0$ (referring to Lebesgue measure) 
  such that $p_{\psi}(x) = \chi_{S}(x)$, $x \in [0,1]$. \\    
  \item \textit{Bessel} (\ref{E:4.7}): $p_{\psi} \in L^{\infty}(0,1)$ and 
  $\|p_{\psi}\|_{\infty} \leq B$. \\
  \item \textit{Frame} $(A,B)$ (\ref{E:4.8}): The estimate 
  $A \leq p_{\psi} \leq B$ a.e. on esssupp $(p_{\psi})$, i.e.,  \\
  holds on esssupp$(p_{\psi})$ where esssupp stands for the essential 
  support of $x \to p_{\psi}(x)$ on $[0,1]$.  \\
  \item \textit{Riesz} (\ref{E:4.9}): The estimate 
  $A \leq p_{\psi} \leq B$ holds for almost every $x \in [0,1]$.   
\end{itemize}
\end{cor}
  
\begin{rem}
\label{R:4.5}
So the distinction between the two conditions (\ref{E:4.8}) and 
(\ref{E:4.8}) is the question of whether the pointwise estimates are 
assumed only on the essential support, or everywhere; of course 
excepting Lebesgue measure $0$.  
\end{rem}

\begin{proof}
The details of proofs are contained in the previous discussion except for 
the necessity of the condition stated in (\ref{E:4.6}) above.

Suppose the Parseval indentity (\ref{E:4.6}) holds on $\mathcal{H}(\psi)$.  
Substituting $f = m(T)\psi$ into (\ref{E:4.6}) we get the following 
identity 
\[
  \int_{0}^{1}|m(x)|^{2}p_{\psi}(x)^{2}dx=\int_{0}^{1}|m(x)|^{2}p_{\psi}(x)dx
  \text{, for all } m.
\]
As a result $p_{\psi}(x)^{2}=p_{\psi}(x)$ for almost every $x \in [0,1]$ or
$p_{\psi}(x)(p_{\psi}(x)-1)=0$, almost every $x \in [0,1]$.
Since $\psi \neq 0$ in $\mathcal{H}$, Lemma \ref{L:3.1} shows that if 
$S = S_{\psi}$ denotes the essential support of $p_{\psi}$, then 
$|S_{\psi}| > 0$, and
\begin{equation}
\label{E:4.21}
  p_{\psi}(x)=\chi_{S_{\psi}}(x) \text{ for almost every } x \in [0,1],
\end{equation}
as claimed.
\end{proof}

\begin{cor}
\label{C:4.5}
Let $T: \mathcal{H} \to \mathcal{H}$ be a unitary operator, 
$\psi \in \mathcal{H} \setminus \{0\}$, and $d\mu_{\psi}(x)=p_{\psi}(x)dx$ as 
above.  We assume that $p_{\psi} \in L^{\infty}(0,1)$ with Bessel bound 
$B(<\infty)$. 

Then the infinite series $\sum_{k \in \mathbb{Z}} c_{k}\psi_{k}$ is well 
defined and norm-convergent in $\mathcal{H}$ for all 
$c=(c_{k})_{k \in \mathbb{Z}} \in l^{2}(\mathbb{Z})$.  Moreover,
\begin{equation}
\label{E:4.22}
  \sum_{k \in \mathbb{Z}} c_{k}\psi_{k} = 0 \text{, for some } 
  c \in l^{2} \setminus (0)
\end{equation}
if and only if there is a 
\begin{equation}
\label{E:4.23}
  E \in \mathcal{B}([0,1]) \text{ such that } |E|>0, \text{ and }
  p_{\psi}=0 \text{ on } E.
\end{equation}
\end{cor}
\begin{proof}
For $c \in l^{2}$, set $m_{c}(x)=\sum_{k \in \mathbb{Z}} c_{k}e_{k}(x)$.
Then
\begin{equation}
\label{E:4.24}
  \int_{0}^{1}|m_{c}(x)|^{2}dx=\sum_{k \in \mathbb{Z}} |c_{k}|^{2}; 
\end{equation}
in particular $m_{c} \in L^{2}(0,1)$.

By Corollary \ref{C:4.4}, functional calculus and Lemma \ref{L:3.1}, the 
vector $m_{c}(T)\psi$ is then well defined for all $c \in l^{2}$; and 
\begin{align*}
  \|m_{c}(T)\|_{\mathcal{H}}^{2} 
  &= \left\Vert \sum_{k \in \mathbb{Z}} c_{k}\psi_{k}\right\Vert_{\mathcal{H}}^{2} \\
  &= \int_{0}^{1}|m_{c}(x)|^{2}p_{\psi}(x)dx  \\
  &\underset{by (\ref{E:4.7})}{\leq} B\int_{0}^{1}|m_{c}(x)|^{2}dx \\
  &\underset{by (\ref{E:4.24})}{=} B\sum_{k \in \mathbb{Z}} |c_{k}|^{2}.
\end{align*}
Hence 
\[
  \lim_{n, m \to \infty}\left\Vert\sum_{k \in (-\infty, -m) \cup (n, \infty)}
c_{k}\psi_{k}\right\Vert_{\mathcal{H}}^{2} = 0, 
\]
proving the first assertion. 

Now suppose $c \in l^{2} \setminus (0)$, and (\ref{E:4.22}) holds.  Then 
\[
  \int_{0}^{1}|m_{c}(x)|^{2}p_{\psi}(x)dx = 0.
\]
If $E_{c}:= \text{esssupp}|m_{c}(x)|^{2}$, then $|E_{c}|>0$, and 
$p_{\psi}=0$ almost everywhere on $E_{c}$.  This is the desired assertion 
(\ref{E:4.23}).

Conversely, suppose (\ref{E:4.23}) holds for some set 
$E \in \mathcal{B}([0,1])$, i.e., that $|E|>0$, and $p_{\psi}=0$ on $E$.
Then, by (\ref{E:4.24}) applied to $\chi_{E}$, we get the 
$L^{2}(0,1)$-expansion 
$\chi_{E}(x) = \sum_{k \in \mathbb{Z}} c_{k}e_{k}(x)$ almost everywhere on 
$[0,1]$, 
$|E|=\|\chi_{E}\|_{L^{2}}^{2}=\sum_{k \in \mathbb{Z}} |c_{k}|^{2}>0$.
Hence $\sum_{k \in \mathbb{Z}} c_{k}\psi_{k} = 0$ in $\mathcal{H}$, 
and $c \neq 0$ in $l^{2}$.  This is (\ref{E:4.22}), and the proof is completed.
\end{proof}

\begin{rem}
\label{R:4.6}
The equivalence of the two assertions (\ref{E:4.22}) and (\ref{E:4.23}) 
in Corollary \ref{C:4.5} above was conjectured in a talk in Chicago by 
Guido Weiss in the special case when the Hilbert space is $L^{2}(\mathbb{R})$ 
and when the unitary operator $T$ is translation by $1$. 

Moreover, in this case, if $\psi$ is the given non-zero 
$L^{2}(\mathbb{R})$-function and if (\ref{E:4.22}) holds for some 
non-zero $c$ in the sequence space $l^{2}$, then we say that the 
integral translates of $\psi$ satisfy an $L^{2}$-dependency.

Since there is typically a substantial amount of redundancy in the 
representation in $L^{2}(\mathbb{R})$ of the sequence of translates  
$\psi_{k}$ , for $k \in \mathbb{Z}$, this notion of $L^{2}$-linear 
dependency, is more subtle than the familiar notion of linear 
dependency (involving only finite sums) from linear algebra. It was 
introduced by Kolmogorov, and it is used in prediction theory, 
see e.g., \cite{MiSa80}.
\end{rem}

\begin{cor}
\label{C:4.7}
Let $T:\mathcal{H} \to \mathcal{H}$ be a unitary operator, 
$\psi \in \mathcal{H} \setminus \{0\}$, and suppose there is some set 
$S \in \mathcal{B}([0,1])$ such that
\begin{equation}
\label{E:4.25}
  0< |S| <1, \text{ and }
\end{equation}
\begin{equation}
\label{E:4.26}
  p_{\psi}(x) = \chi_{S}(x), \text{ for all } \in [0,1].
\end{equation}
Then there is some $c \in l^{2} \setminus (0)$ such that 
$\sum_{k \in \mathbb{Z}} c_{k}T^{k}\psi = 0$, i.e., there is a non-trivial 
linear relation in $\mathcal{H}(\psi)$.
\end{cor}
\begin{proof}
Note that (\ref{E:4.25})(\ref{E:4.26}) and Corollary \ref{C:4.4} (\ref{E:4.6}) 
imply that the vectors $\{T^{k}\psi| k \in \mathbb{Z}\}$ form a Parseval 
frame in $\mathcal{H}(\psi)$; and moveover that 
\begin{equation}
\label{E:4.27}
  \|m(T)\psi\|_{\mathcal{H}}^{2} = \int_{S}|m(x)|^{2}dx
\end{equation}
holds for all $m \in L^{2}(0,1)$.  Now take $m:= \chi_{[0,1] \setminus S}$, 
and let $(c_{k}) \in l^{2} \setminus (0)$ be the corresponding Fourier 
coefficients; i.e., 
\[
  1-|S| = \sum_{k \in \mathbb{Z}} |c_{k}|^{2} >0.
\]
Recall $m(x)=\sum_{k \in \mathbb{Z}}c_{k}e_{k}(x)$, and 
$\|m(x)\|_{L^{2}}^{2}=\sum_{k \in \mathbb{Z}}|c_{k}|^{2}$.

But then $\sum_{k \in \mathbb{Z}} c_{k}T^{k}\psi = 0$ in $\mathcal{H}$ by 
virtue of (\ref{E:4.27}).
\end{proof}

\section{Dyadic Wavelets}
\label{sec:5a}
In this section we read off some corollaries regarding the two spectral 
functions that describe dyadic wavelets in $L^{2}(\mathbb{R})$.
\begin{lem}
\label{L:dl}
We state the following result for dyadic wavelet functions i.e., 
\[
  m_{0}: \mathbb{R}/\mathbb{Z} \to \mathbb{C}, \\
  |m_{0}(t)|^{2} + |m_{0}(t+\frac{1}{2})|^{2}=1,
\]
where $|m_{0}(0)|=1$, the filter $m_{0}$ is assumed to satisfy the 
low-pass condition, passing at $t = 0$.

Let $\varphi, \psi \in L^{2}(\mathbb{R})$ be the two functions for a wavelet 
filter $m_{0}$.  The consistency relation is the following identity which 
holds for all $t$:
\[
  p_{\varphi}(t)+p_{\psi}(t)=p_{\varphi}(\frac{t}{2})+p_{\varphi}(\frac{t+1}{2})
\]
\end{lem}
\begin{proof}
\cite{BJMP05}.
\end{proof}

\begin{center}
\begin{tabular}{c}
Wavelet functions
\end{tabular}
\begin{tabular}{l|l}
\hline
Father function $\varphi$  & Mother function $\psi$\\
\hline
$\varphi(x)= 1$ on $[0,1]$ and $0$ in $\mathbb{R} \setminus [0,1]$ & 
$\psi(x)=1$ in $[0, \frac{1}{2})$, $-1$ in $[\frac{1}{2}, 1)$, and
$0$ in $\mathbb{R} \setminus [0,1)$. \\
\hline
\end{tabular}
\end{center}

Given $k$ an odd integer, we now describe the STRETCHED HAAR WAVELET 

\begin{center}
\begin{tabular}{l|l}
  \\
  $\varphi_{k}=\frac{1}{k}\varphi(\frac{x}{k})$ \text{   } & \text{   } $\psi_{k}=\frac{1}{k}\psi(\frac{x}{k})$ \\
  \\
  $|\widehat{\varphi_{k}}(t)|^{2}=\frac{sin^{2}(\pi kt)}{k^{2}(\pi t)^{2}}$ \text{   } & \text{   }  
  $|\widehat{\psi_{k}}(t)|^{2}=\frac{sin^{4}(\frac{\pi kt}{2})}{k^{2}(\pi t)^{2}}$ \\ \\
  $\text{PER}|\widehat{\varphi}|^{2}= \frac{1}{k^{2}}(\frac{sin(\pi kt)}{sin(\pi t)})^{2}$ \text{   } & \text{   }  
  $\text{PER}|\widehat{\psi}|^{2}= \frac{1}{(2k)^{2}}
(\frac{sin^{4}(\frac{\pi kt}{2})}{sin^{2}(\frac{\pi t}{2})}
+\frac{cos^{4}(\frac{\pi kt}{2})}{cos^{2}(\frac{\pi t}{2})})$. \\
  \\
\end{tabular}
\end{center}

Where we use the formula
\[
  \sum_{n \in \mathbb{Z}} \frac{1}{(t+n)^{2}} = \frac{\pi^{2}}{sin^{2}(\pi t)}
  = \pi^{2}csc^{2}(\pi t), \text{    } t \in \mathbb{R}.
\]

\subsection*{Spectral Density Functions on $[0,1] \simeq \mathbb{R}/\mathbb{Z}$}
\begin{center}
\begin{tabular}{l l}
  \\
  $p_{\varphi_{1}}(t) \equiv 1$ \text{   } & \text{   } $p_{\psi_{1}}(t) \equiv 1$  \\ \\
  If $k>1$ then & \text{ } \\
  $p_{\varphi_{k}}(t)= \frac{1}{k^{2}}(\frac{sin(\pi kt)}{sin(\pi t)})^{2}$ \text{   } & \text{   } $p_{\psi_{k}}(t)= \frac{1}{(2k)^{2}}
(\frac{sin^{4}(\frac{\pi kt}{2})}{sin^{2}(\frac{\pi t}{2})}
+\frac{cos^{4}(\frac{\pi kt}{2})}{cos^{2}(\frac{\pi t}{2})})$. \\
  \\
\end{tabular}
\end{center}

\begin{figure}[htb]
\label{F:pphi3}
  \begin{center}
    \includegraphics[width=3.3in]{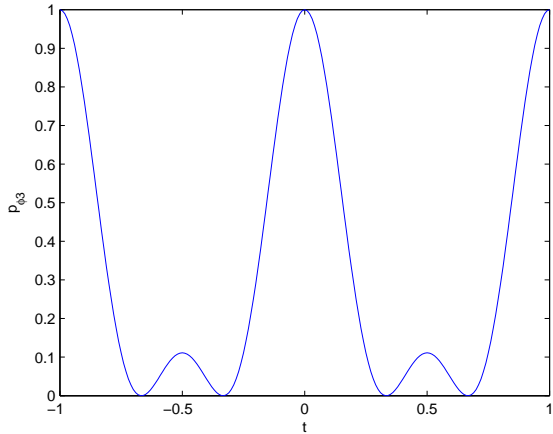}
    \caption{Graph of $p_{\varphi_{3}}$.}
  \end{center}
\end{figure}


\begin{figure}[htb]
\label{F:ppsi3}
  \begin{center}
    \includegraphics[width=3.3in]{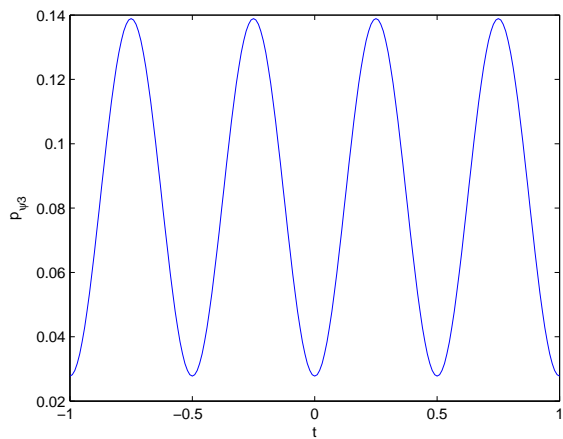}
    \caption{Graph of $p_{\psi_{3}}$.}
  \end{center}
\end{figure}

\begin{figure}[htb]
\label{F:pphipsi3}
  \begin{center}
    \includegraphics[width=3.3in]{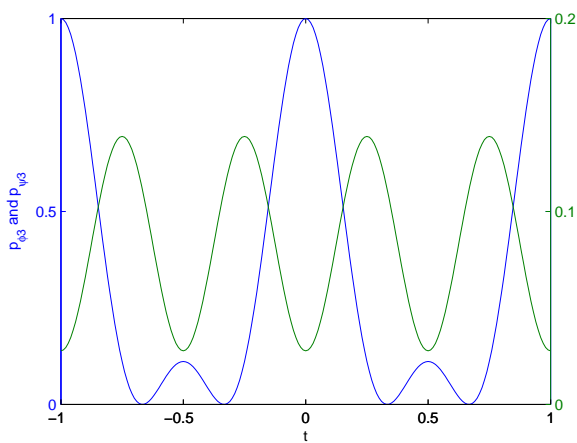}
    \caption{Graph of $p_{\varphi_{3}}$ and $p_{\psi_{3}}$.}
  \end{center}
\end{figure}


Hence we have consistency, but the peculiar thing is that in the stretched Haar 
case, we have a Parseval wavelet (i.e., for the system $(\psi_{j,k})$, i.e., 
$\psi_{j,k}(x)=2^{j/2}\psi(2^{j}x-k)$, $j,k \in \mathbb{Z}$, built 
on the interval from $0$ to $3$, instead of from $0$ to $1$). Specifically, 
$\sum_{j}\sum_{k}|\langle \psi_{j,k}|f\rangle|^{2}= \|f\|^{2}$, for all 
$f \in L^{2}(\mathbb{R})$ but there is no Parseval property on 
$\mathcal{H}(\varphi)$ or on $\mathcal{H}(\psi)$. 
Indeed, from \cite{Jor06}, Chapter 6 we know that the mother function 
$\psi$ generates 
a Parseval wavelet frame in $\mathcal{H}:= L^{2}(\mathbb{R})$, but only if 
we use both dyadic scaling and $\mathbb{Z}$-translations. However, in the 
$3$-stretched Haar case, we don't get Parseval frame systems for 
$\mathbb{Z}$-translations alone.

Specifically, if we consider the stretched father and mother functions, by 
themselves, i.e., the functions $\varphi$ and $\psi$  individually, then 
one may ask whether or not the system $\{\varphi(\cdot + n) | n \in \mathbb{Z}\}$ 
is Parseval in $\mathcal{H}(\varphi)$; or 
$\{\psi(\cdot + n) | n \in \mathbb{Z}\}$ in $\mathcal{H}(\psi)$, and the 
answer is `no' for both. Neither of the two has the Parseval property. Using 
our Corollary, and our computations of the two spectral functions $p_{\psi}$ 
and $p_{\varphi}$ we can only ascertain that the two $\mathbb{Z}$-translation 
systems have the Bessel property with Bessel constant $B = 1$, but neither 
satisfies the stronger ones of the frame properties.

\section{Cyclic Subspaces}
\label{sec:5b}
\begin{defn}
\label{D:closed}
Let $\mathcal{A}$ be an abelian group (algebra) represented by bounded 
operators on a Hilbert space $\mathcal{H}$.  A closed subspace 
$\mathcal{K} \subset \mathcal{H}$ is \textit{cyclic} if there exists 
$\varphi \in \mathcal{K} \setminus \{0\}$ such that 
$[\mathcal{A}\varphi]=\mathcal{K}$.  Here $[\cdot]$ denotes closed linear 
span.
\end{defn}

\begin{lem}
\label{L:4.8b}
Let $\mathcal{A}$ be as in definition \ref{D:closed} acting on a given 
Hilbert space $\mathcal{H}$.  Then there is an indexed family of cyclic 
vectors $(\varphi_{j})$ such that 
\begin{equation}
\label{E:oplus}
  \mathcal{H} = \bigoplus_{j} [\mathcal{A}\varphi_{j}]
\end{equation}
The symbol $\bigoplus$ in (\ref{E:oplus}) indicates orthogonality, i.e.,
\[
  [\mathcal{A}\varphi_{j}] \perp [\mathcal{A}\varphi_{k}] \text{ if } j \neq k.
\]
\end{lem}
\begin{proof}
This is a standard appliation of Zorn's lemma, see e.g., \cite{Nel69}
\end{proof}

\begin{cor}
\label{C:4.8}
Let $T:\mathcal{H} \to \mathcal{H}$ be a unitary operator, and suppose 
$p_{\psi} \in L^{\infty}$ for all $\psi \in \mathcal{H}$.
Let $\mathcal{M} \subset \mathcal{H}$ be a closed $T$-invariant subspace.  
Then there are subsets $[0,1] \supset S_{1} \supset S_{2} \supset \cdots$ 
in $\mathcal{B}([0,1])$, and vectors $\psi_{j} \in \mathcal{M}$ such that 
$p_{\psi_{j}} = \chi_{S_{j}}$ for $j=1,2, \cdots$, and
\begin{equation}
\label{E:4.28}
  \mathcal{M} = \bigoplus_{j \geq 1} \mathcal{H}(\psi_{j})
\end{equation}
where the summation in (\ref{E:4.28}) indicates orthogonality of the 
spaces $\mathcal{H}(\psi_{j})$ corresponding to different index values 
for $j$.
\end{cor}
\begin{proof}
This follows from the Spectral-Theorem applied to the restricted operator 
$T |_{\mathcal{M}}$, combined with the previous two corollaries; see also 
the references \cite{BMM99, Bag00} for special cases of this result.
\end{proof}

\section{Stochastic Processes}
\label{sec:5}
In the next section we show that our spectral analysis of translations 
produces a particular class of stochastic processes. Hence our results for 
translations, and more generally, for unitary operators $T$ in Hilbert 
space have a probabilistic significance. 

This is relevant to the way wavelets are used in the processing of signals 
with noise; see \cite{Jor06, Jor06b} and the following papers 
\cite{JoSo07, Son07}.
A popular method of spectral analyzing correlations in stochastic 
components in a signal or in an image is to introduce a spectral function, 
or spectral kernel, and then an associated selfadjoint operator. When this 
operator is then diagonalized, a substantial simplification occurs. The 
diagonalization goes under the name ``Karhunen-Lo\`{e}ve;"  
see \cite{JoSo07, Son07}. But for the use of the Karhunen-Lo\`{e}ve transform 
method, 
it will be significant that the stochastic processes involved be Gaussian. 
In this section we show that there is always such a Gaussian choice, and 
that it is canonical. Moreover we show (Corollary \ref{C:5.6}) that each of these 
Gaussian processes is naturally associated with a unitary operator so our 
results from sections \ref{sec:3} and \ref{sec:4} apply.

A stochastic process is an indexed system of random variables. By ``random 
variables" we mean measurable functions on a probability space 
$(\Omega, \mathcal{M}, P)$ where $\Omega$ is a set, $\mathcal{M}$ is a fixed 
sigma-algebra of subsets in $\Omega$ (events), and $P$ is a probability 
measure, i.e., $P(S)$ is defined for $S \in \mathcal{M}$.

We will make the starting assumption that our random variables $X$, 
referring to $(\Omega, \mathcal{M}, P)$, will be in 
$L^{2}(\Omega, \mathcal{M}, P)$, abbreviated $L^{2}(P)$, i.e., that 
\begin{equation}
\label{E:5.1}
  \int_{\Omega} |X(\omega)|^{2}dP(\omega) = E_{\omega}(|X(\omega)|^{2})
  = \langle X | X \rangle_{L^{2}(P)} < \infty.
\end{equation}

In the present discussion, the index set for the random variables will be 
$[0,1]$ or the Borel sets $\mathcal{B}([0,1])$.  We will say that a process 
$X$ is Gaussian if $X_{t}$, for $t \in [0,1]$, or $X_{A}$, for 
$A \in \mathcal{B}([0,1])$, is Gaussian.

\begin{lem}
\label{L:5.1}
(Kolmogorov) Let $W$ be a set and let $R : W \times W \to \mathbb{C}$ be a 
\textit{positive (semi)-definite} function, i.e., 
\begin{equation}
\label{E:5.2}
  \sum \sum \overline{c}_{v}R(v,w)c_{w} \geq 0 \text{,     } (v,w) 
  \in W \times W,
\end{equation}
for all finite sequences $(c_{w \in W})$, i.e., 
$c_{w} \in \mathbb{C} \setminus (0)$ for at most a finite subset in 
$W$, depending on $c$.

Then there is a Gaussian stochastic process $X, W, (\Omega, \mathcal{M}, P)$ 
such that 
\begin{equation}
\label{E:5.3}
  E(\overline{X}_{v}X_{w}) = R(v,w) \text{, for all } (v,w) \in W \times W.
\end{equation}
\end{lem}
\begin{proof}
See e.g., \cite{Jor06} or \cite{Nel69}.
\end{proof}

\begin{ex}
\label{ex:5.2}
(Brownian motion.) For $s, t \in [0,1]$, set $s \wedge t = min(s,t)$, and 
\begin{equation}
\label{E:5.4}
  R(s,t) := s \wedge t.
\end{equation}
Then $R$ is positive definite, and the Gaussian process $(X_{t})$ with 
\begin{equation}
\label{E:5.5}
  E(\overline{X}_{s}X_{t}) = s \wedge t \text{, and}
\end{equation}
\begin{equation}
\label{E:5.6}
  E(X_{t}) = 0
\end{equation}
is called the normalized Brownian motion on $[0,1]$.
\end{ex}

\begin{ex}
\label{ex:5.3}
(Set-indexed Gaussian.) Let $\mu$ be a finite Borel measure on $[0,1]$. 
For $A, B \in \mathcal{B}([0,1])$, the Borel sets in the unit interval, set 
\begin{equation}
\label{E:5.7}
  R(A,B):= \mu(A \cap B).
\end{equation}

Then $R$ is positive definite, and the Gaussian process 
$(X_{A})_{A \in \mathcal{B}([0,1])}$ with
\begin{equation}
\label{E:5.8}
  E(\overline{X}_{A}X_{B}) = \mu(A \cap B)
\end{equation}
and $E(X_{A}) = 0$, $A, B \in \mathcal{B}([0,1])$; is called the 
$\mu$-Gaussian process.  (Note that $E(\cdot)=\int_{\mathbb{R}}\cdot dP$, 
and that $P$ depends on $\mu$.)
\end{ex}

\begin{rem}
\label{R:5.4}
An easy computation shows that Brownian motion is $\mu$-Gaussian with 
$\mu :=$ the Lesbegue measure restricted to $[0,1]$.
\end{rem}

\begin{lem}
\label{L:5.5}
\cite{Nel69} Let $(X_{A})_{A \in \mathcal{B}([0,1])}$ be a $\mu$-Gaussian 
process; and let $m \in L^{2}([0,1], \mu)$.  Then the stochastic integral 
\begin{equation}
\label{E:5.9}
  \int_{0}^{1}m(t)dX_{t} \in L^{2}(\Omega, \mathcal{M}, P) 
\end{equation}
and
\begin{equation}
\label{E:5.10}
  \left\Vert\int_{0}^{1}m(t)dX_{t}\right\Vert^{2}_{L^{2}(P)} = \int_{0}^{1}|m(t)|^{2}d\mu(t). 
\end{equation}
\end{lem}
\begin{proof}
The idea in the proof is to construct the Gaussian probability space 
$(\Omega, \mathcal{M}, P)$ as in Lemma \ref{L:1.1}.  The main difference 
between the two cases is that our stochastic processes $X$ is now indexed 
by Borel sets in contrast to the discrete index used in Lemma \ref{L:1.1}.
\end{proof}

\begin{cor}
\label{C:5.6}
Let $(X_{A})_{A \in \mathcal{B}([0,1])}$ be a $\mu$-Gaussian process; and 
define an operator $T:L^{2}(P) \to L^{2}(P)$ by 
\begin{equation}
\label{E:5.11}
  T(\int_{0}^{1}m(t)dX_{t}) := \int_{0}^{1}e_{1}(t)m(t)dX_{t},
\end{equation}
then $T$ is unitary.
\end{cor}
\begin{proof}
By (\ref{E:5.10}), it is enough to prove that 
\begin{equation}
\label{E:5.12}
  L^{2}(\mu) \ni m(t) \mapsto e_{1}(t)m(t) \in L^{2}(\mu),
\end{equation}
is a unitary operator; but this last fact is obvious.
\end{proof}

\begin{ex}
\label{ex:5.7}
(Non-Gaussian) Let $(Tf)(x) := f(x-1)$ be translation in $L^{2}(\mathbb{R})$, 
and let $\psi \in L^{2}(\mathbb{R}) \setminus \{0\}$.  Set 
\begin{equation}
\label{E:5.13}
  d\mu_{\psi}(t):=p_{\psi}(t)dt
\end{equation}
with
\begin{equation}
\label{E:5.14}
  p_{\psi}(t):= \sum_{n \in \mathbb{Z}}|\widehat{\psi}(t+n)|^{2}, \text{  }
  t \in [0,1];
\end{equation}
and 
\begin{equation}
\label{E:5.15}
  p_{\psi}(s,t):= \sum_{n \in \mathbb{Z}}\overline{\widehat{\psi}(s+n)}\widehat{\psi}(t+n).
\end{equation}
\end{ex}
Set
\begin{equation}
\label{E:5.16}
  X_{t}(n):= \widehat{\psi}(t+n), \text{   } n \in \mathbb{Z} \text{,   } 
  t \in [0,1].
\end{equation}

Then each $X_{t}$ in $(X_{t})_{t \in [0,1]}$ is a random variable on 
$\mathbb{Z}$ with counting measure, and 
\begin{equation}
\label{E:5.17}
  E(\overline{X}_{s}X_{t}) = p_{\psi}(s,t) \text{, for all } s, t \in [0,1].
\end{equation}
\begin{proof}
See section \ref{sec:3} above.
\end{proof}

Our final result also follows from these considerations and the lemmas 
in section \ref{sec:3}.  It is about the special case when function 
$p_{\psi}(\cdot)$ in (\ref{E:5.14}) is in $L^{2}([0,1])$.

\begin{rem}
\label{R:5.5}
It is of interest to compare the realization of $X_{t}(n):=\widehat{\psi}(t+n)$ 
in $L^{2}(\mathbb{R})$; see (\ref{E:5.16}) with its Gaussian version in 
$L^{2}(\Omega_{\psi}, P_{\psi})$, i.e., with the stochastic integral
\[
  \int_{0}^{1}m(t)dX_{t} \in L^{2}(\Omega_{\psi}, P_{\psi}),
\]
and with $P_{\psi}$ denoting the Gaussian measure.

The result is as follows:
\newline
For $s, t \in [0,1]$, set 
\[
  Q^{\psi}(s,t):=\sum_{n \in \mathbb{Z}} \overline{\widehat{\psi'}(s+n)}
\widehat{\psi'}(t+n),
\]
and $\widehat{\psi'}(t):=\frac{d}{dt}\widehat{\psi}(t)$.
Then 
\[
  \left\Vert\int_{0}^{1}m(t)dX_{t}\right\Vert^{2}_{L^{2}(\Omega_{\psi}, P_{\psi})}
  = \int_{0}^{1}\int_{0}^{1}\overline{m(s)}m(t)Q^{\psi}(s,t)dsdt.
\]
\end{rem}

\begin{cor}
\label{C:5.6a}
Let $T, L^{2}(\mathbb{R})$ and $\psi$ be as in Example (Non-Gaussian) above. 
Set
\begin{equation}
\label{E:5.18}
  \psi_{k} := T^{k}\psi = \psi(\cdot -k), \text{ } k \in \mathbb{Z}.
\end{equation}

On the dense subspace $\mathcal{D} \subset l^{2}$ of finitely indexed 
sequences, set
\begin{equation}
\label{E:5.19}
  F((c_{k})_{k \in \mathbb{Z}}):= \sum_{k \in \mathbb{Z}}c_{k}\psi_{k}.
\end{equation}

We have the following (a) $\Longleftrightarrow$ (b) $\Longrightarrow$ (c), 
where the three affirmations (a)-(c) are as follows:
\begin{itemize}
  \item[(a)] $p_{\psi} \in L^{2}(0,1)$,
  \item[(b)] $(\langle \psi | \psi_{k} \rangle) \in l^{2}(\mathbb{Z})$,
  and
  \item[(c)] $l^{2} \supset \mathcal{D} \overset{F}{\longrightarrow}
  \mathcal{H}(\psi) \subset L^{2}(\mathbb{R})$ is a closable operator, when 
  $F$ is defined by (\ref{E:5.19}).
\end{itemize}
\end{cor}
\begin{proof}
Computing the Fourier coefficients $\widehat{p_{\psi}}(k)$, $k \in \mathbb{Z}$, 
we find 
\begin{align*}
  \widehat{p_{\psi}}(k) &= \int_{0}^{1}\overline{e_{k}(t)}p_{\psi}(t)dt \\
  &= \int_{0}^{1}\overline{e_{k}(t)}\sum_{n \in \mathbb{Z}}|\widehat{\psi}(t+k)|^{2}dt \\
  &= \int_{\mathbb{R}}\overline{e_{k}(t)}|\widehat{\psi}(t)|^{2}dt \\
  &\underset{\text{by Parseval}}{=} \int_{\mathbb{R}}\overline{\psi(x)}
  \psi(x-k)dx  \\
  &= \langle \psi | \psi_{k} \rangle_{L^{2}}, \text{ } k \in \mathbb{Z}, 
\end{align*}
which proves (a) $\Longleftrightarrow$ (b).

Supposing (a), to prove that $F$ is a closable operator, note the formula 
\[
  (F^{*}f)_{k}=\int_{\mathbb{R}}f(x-k)dx, \text{ } k\in \mathbb{Z}
\]
for the adjoint operator.  So condition (a) is the assertion that each 
$\psi_{j}$ is in the domain of $F^{*}$.  Since 
$(\psi_{j})_{j \in \mathbb{Z}}$ spans a dense subspace in 
$\mathcal{H}(\psi) \subset L^{2}(\mathbb{R})$, it follows that $F$ in 
(\ref{E:5.19}) is a closable operator.
\end{proof}

\section{A Connection to Karhunen-Lo\`{e}ve Transforms}
\label{sec:6}
Suppose $X_{t}$ is a stochastic process indexed by $t$ in a finite interval 
$J$, and taking values in $L^{2}(\Omega, P)$ for some probability space 
$(\Omega, P)$. Assume the normalization $E(X_{t})=0$. Suppose the integral 
kernel $E(\overline{X_{t}} X_{s})$ can be 
diagonalized, i.e., suppose that 
\[
  \int_{J}{E(\overline{X_{t}} X_{s})\varphi_{k}(s)}ds=\lambda_{k}\varphi_{k}(t)
\]
with an ONB $(\varphi_{k})$ in $L^{2}(J)$.  If $E(X_{t})=0$ then
\[
  X_{t}(\omega)=\sum_{k}\sqrt{\lambda_{k}}\varphi_{k}(t)Z_{k}(\omega), 
  \quad \omega \in \Omega
\]
where $E(\overline{Z_{j}} Z_{k})=\delta_{j,k}$, and $E(Z_{k})=0$.
The ONB $(\varphi_{k})$ is called the \textit{KL-basis} with respect to the 
stochastic processes $\{X_{t}: t \in I \}$.

The Karhunen-Lo\`{e}ve-theorem \cite{Ash90} states that if $(X_{t})$ is 
Gaussian, then so are the random variables $(Z_{k})$.  Furthermore, they 
are $N(0,1)$ i.e., normal with mean zero and variance one, so independent 
and identically distributed. This last fact explains the familiar 
\textit{optimality} of Karhunen-Lo\`{e}ve method in transform coding.

The following result illustrates the significance of our spectral density 
functions in the analysis of Karhunen-Lo\`{e}ve transforms.

\begin{thm}
\label{T:6.1}
Let $(\Omega, P)$ by a probability space, $J \subset \mathbb{R}$ an interval 
(possibly infinite), and let $(X_{t})_{t\in J}$ be a stochastic process with 
values in $L^{2}(\Omega, P)$.  Assume $E(X_{t})=0$ for all $t \in J$.  Then 
by the spectral density theorem for unitary operator 
(see Corollary \ref{C:4.8}), the Hilbert space 
$L^{2}(J)$ splits as an orthogonal sum
\begin{equation}
\label{E:Lsum}
  L^{2}(J)=\mathcal{H}_{d}\oplus \mathcal{H}_{c}
\end{equation} 
(d is for discrete and c is for continuous) such that the following data 
exists:
\begin{itemize}
  \item[(a)] $(\varphi_{k})_{k \in \mathbb{N}}$ an ONB in $\mathcal{H}_{d}$.
  \item[(b)] $(Z_{k})_{k \in \mathbb{N}}$ : independent random variables.
  \item[(c)] $E(\overline{Z_{j}} Z_{k})=\delta_{j,k}$, and $E(Z_{k})=0$.
  \item[(d)] $(\lambda_{k}) \subset \mathbb{R}_{\geq 0}$.
  \item[(e)] $\varphi(\cdot,\cdot)$ : a Borel measure on $\mathbb{R}$ in the 
first variable, such that 
  \begin{itemize}
    \item[(i)] $\varphi(A, \cdot) \in \mathcal{H}_{c}$ for $A$ an open 
subinterval of $J$,
  \end{itemize} and
  \begin{itemize}
    \item[(ii)] $\langle\varphi(A_{1}, \cdot)| \varphi(A_{2}, \cdot)
    \rangle_{L^{2}(J)} =0$ whenever $A_{1} \cap A_{2} = \emptyset$.
  \end{itemize}
  \item[(f)] $Z(\cdot, \cdot)$ : a measurable family of random variables such
  that $Z(A_{1}, \cdot)$ and $Z(A_{2}, \cdot)$ are independent when
  $A_{1}, A_{2} \in \mathcal{B}_{J}$ and $A_{1} \cap A_{2} = \emptyset$, 
  \[
    E(\overline{Z(\lambda, \cdot)} Z(\lambda', \cdot))=\delta(\lambda-\lambda'), 
    \text{ and } E(Z(\lambda, \cdot))=0.
  \]
\end{itemize}
Finally, we get the following Karhunen-Lo\`{e}ve expansions for the 
$L^{2}(J)$-operator with integral kernel $E(X_{t} X_{s})$:
\begin{equation}
\label{E:Suml}
  \sum_{k \in \mathbb{N}} \lambda_{k}|\varphi_{k} \rangle \langle \varphi_{k}|
  + \int_{J}{\lambda|\varphi(d \lambda, \cdot)\rangle \langle 
    \varphi(d \lambda, \cdot)|}
\end{equation}
Moreover, the process decomposes thus:
\begin{equation}
\label{E:Sumsqrtl}
  X_{t}(\omega)= \sum_{k\in \mathbb{N}}\sqrt{\lambda_{k}}Z_{k}(\omega)
  \varphi_{k}(t)+\int_{J}{\sqrt{\lambda}Z(\lambda, \omega)\varphi(d\lambda, t)}. 
\end{equation}
\end{thm}








\bibliographystyle{alpha}
\bibliography{optdec}


\subsection*{Acknowledgment}
The authors are pleased acknowledge helpful discussions
with Professor Guido Weiss, and the members of the Washington University
Wavelet Seminar. Our work was inspired by lectures of Guido Weiss,
including at the regional AMS meeting (Chicago, October 2007.)
\end{document}